\documentclass[a4paper,11pt]{amsart}
\usepackage{hyperref,fancyhdr,mathrsfs,amsmath,amscd,amsthm,amsfonts,latexsym,amssymb,stmaryrd}
\usepackage[all]{xy}

\DeclareMathOperator{\trace}{trace}

\DeclareMathOperator{\rank}{rank}

\DeclareMathOperator{\dif}{d}

\newcommand{\Cal}{\mathcal{C}}

\newcommand{\Fa}{\mathcal{F}}

\newcommand{\ol}{\mathcal{O}}

\def \a{\alpha}

\def \b{\beta}

\def \g{\gamma}

\def \phi{\varphi}
\def \Phi{\varPhi} 
\def \Psi{\varPsi} 
\def \p{\pi}
\def \r{\rho}
\def \s{\sigma}
\def \t{\tau}

\def \R{\mathbb{R}}

\def \C{\mathbb{C}\,}

\def\widecheckg{g^{\hspace*{-2.5pt}\vbox to 5pt{\hbox to
0pt{\LARGE$\check{}$}}}\hspace*{2pt}}

\def\widecheckl{\lambda^{\hspace*{-3.5pt}\vbox to 8pt{\hbox to
0pt{\LARGE$\check{}$}}}\hspace*{2pt}}

\begin{document}

\title{(Pluri)harmonic morphisms\\ 
and the Penrose--Ward transform}
\author{Radu Pantilie}
\email{\href{mailto:radu.pantilie@imar.ro}{radu.pantilie@imar.ro}}
\address{R.~Pantilie, Institutul de Matematic\u a ``Simion~Stoilow'' al Academiei Rom\^ane,
C.P. 1-764, 014700, Bucure\c sti, Rom\^ania}

\newtheorem{thm}{Theorem}[section]
\newtheorem{lem}[thm]{Lemma}
\newtheorem{cor}[thm]{Corollary}
\newtheorem{prop}[thm]{Proposition} 

\theoremstyle{definition}

\newtheorem{defn}[thm]{Definition}
\newtheorem{rem}[thm]{Remark}
\newtheorem{exm}[thm]{Example}

\numberwithin{equation}{section}
 
\begin{abstract} 
We show that, in quaternionic geometry, the Ward transform is a manifestation of the functoriality of the basic correspondence 
between the $\r$-quaternionic manifolds and their twistor spaces. We apply this fact, together with the Penrose transform, to 
obtain existence results for hypercomplex manifolds and for harmonic morphisms from hyper-K\"ahler manifolds. 
\end{abstract} 

\maketitle 
\thispagestyle{empty} 
\vspace{-5mm}

\section*{Introduction} 

\indent 
The Gibbons--Hawking construction \cite{GibHaw} starts with a positive harmonic function $v$\,, on an open subset $D$ of the Euclidean space, 
and builds the hyper-K\"ahler metric $g=vh+v^{-1}(\dif\!t+A)^2$ on $D\times\R$\,, where $h$ is the canonical metric on $D$ and $A$ is a (local) 
one-form on $D$ such that $(v,A)$ is a monopole; that is, $\dif\!v=*\dif\!A$\,.\\ 
\indent 
It is useful to, also, know the following facts:\\ 
\indent 
\quad(1) $(v,A)$ is a monopole if and only if the exterior derivative of $v\dif\!t-A$ is anti-self-dual (this is fairly well known and easy to prove),\\ 
\indent 
\quad(2) the twistor space $Z_{D\times\R}$ of $(D\times\R,g)$ is the holomorphic bundle over the twistor space $Z_D$ of $(D,h)$ ($Z_D$ is an open subset of $\ol(2)$\,) 
corresponding to the monopole $(v,A)$\,, through the Ward transform \cite{Hit-monop_geod}\,,\\  
\indent 
\quad(3) the projection $\phi:(D\times\R)\to(D,h)$ is a twistorial harmonic morphism:\\ 
\indent 
\qquad(3a) $\phi$ pulls back (local) harmonic functions on $(D,h)$ to harmonic functions on $(D\times\R,g)$ (see \cite{BaiWoo2}\,, \cite{Pan-EAR}\,),\\ 
\indent 
\qquad(3b) $\phi$ corresponds to a surjective holomorphic submersion $\Phi:Z_{D\times\R}\to Z_D$ mapping the twistor spheres 
diffeomorphically onto twistor spheres ($\Phi$ is just the projection of $Z_{D\times\R}$ as a bundle over $Z_D$).\\ 
\indent 
\emph{How to generalize this setting to higher dimensions?} A partial solution to this problem is known to exist by replacing $\ol(2)$ with $n\ol(2)$\,, 
where $n\in\mathbb{N}\setminus\{0\}$ (see \cite{PePo-88} and the references therein). 
In this paper, we adopt a more general approach which, in particular, gives an existence result for harmonic morphisms from hyper-K\"ahler manifolds 
whose twistor spaces are built over $Z_D=\ol(k)$\,, where $k\geq2$ is even. 
Briefly, we tackle this problem by observing that the Euclidean space and $\R$ are, just, $U_2$ and $U_0$\,, 
respectively, where $U_k$ is the irreducible representation of dimension $k+1$ of ${\rm SO}(3)$\,, with $k\geq2$ even. Then recall that 
$U_k$ is endowed with a ${\rm SO}(3)$-invariant Euclidean structure $h_k$\,, unique up to homotheties. 
Furthermore, $U_k$ is, also, endowed with a natural structure, which we call $\r$-quaternionic \cite{Pan-qgfs} (see, also, \cite{fq,fq_2,Pan-twistor_(co-)cr_q}\,), 
whose twistor space is $\ol(k)$\,; here, $\r$ is given by the projection from $E=U_k\oplus U_{k-2}$ onto $U_k$\,. 
This way, $D$ can be any open subset of $U_k$\,, and, as $E$ is a quaternionic vector space we, also, have an adequate notion 
of anti-self-dual two-form; that is, a two-form whose $(0,2)$-components with respect to all compatible linear complex structures are zero.\\ 
\indent  
Now, one of the main tasks is to establish the Ward transform for $\r$-quaternionic manifolds. 
In doing this, we realized that, in quaternionic geometry, \emph{the Ward transform is a manifestation of the functoriality of the basic correspondence 
between the $\r$-quaternionic manifolds and their twistor spaces}. Also, \emph{the adequate higher dimensional version of the notion of monopole is 
provided by the anti-self-dual principal $\r$-connections}.\\ 
\indent 
For simplicity, in this paper, we work only in the complex analytic category. In Section \ref{section:ro_quaternionic_vector_spaces} 
we review the $\r$-quaternionic vector spaces. Then in Sections \ref{section:inf_Ward_transform} and \ref{section:Ward_transform} 
we deal with the Ward transform, where the main results are Theorems \ref{thm:infinitesimal_Ward_transform} and \ref{thm:Ward_transform} 
but, along the way, other facts are obtained (such as, Remark \ref{rem:cohomology_for_pluriharmonic}\,) which are important for the intended applications.\\ 
\indent 
In Section \ref{section:WhK} we show that hypercomplex manifolds, obtained through the Ward transform, admit compatible hyper-K\"ahler metrics  
(Theorem \ref{thm:hK_Z_surface}\,). Moreover, these hyper-K\"ahler manifolds are domains of harmonic morphisms (Corollary \ref{cor:hmhK_1}\,). 
Furthermore, to define the relevant `harmonic sheaf' on the codomains of these harmonic morphisms, we introduce a (generalized) Obata $\r$-connection 
whose existence is, also, provided by the Ward transform.\\ 
\indent 
The applications are continued in Section \ref{section:pluriharmonic_and_hypercomplex}\,, where, firstly, we define in a convenient way the relevant notion 
of pluriharmonicity (Definition \ref{defn:pluriharmonic}\,), given by the Penrose transform. 
This is then used to obtain a natural correspondence between hypercomplex manifolds, 
whose twistor spaces are principal bundles with abelian structural groups, and certain cohomology classes (Corollary \ref{cor:hmhK_2}\,).

\section{$\r$-quaternionic vector spaces} \label{section:ro_quaternionic_vector_spaces} 

\indent 
We shall ignore the conjugations; that is, we shall work in the complex analytic category. Thus, all the objects and maps are assumed complex analytic 
and by the tangent bundle of a manifold we shall mean the holomorphic tangent bundle; in particular, all the vector spaces are assumed complex.\\ 
\indent 
Firstly, we formulate the complex versions of two classical notions.  

\begin{defn} 
A \emph{linear hypercomplex structure} on a vector space $E$ is a morphism of (unital) associative algebras from $\mathfrak{gl}(2,\C\!)$ to ${\rm End}\,E$\,.\\ 
\indent 
Two linear hypercomplex structures $\s, \s':\mathfrak{gl}(2,\C\!)\to{\rm End}\,E$ 
are \emph{equivalent} if there exists $a\in{\rm GL}(2,\C\!)$ such that $\s'=\s\circ{\rm Ad}\,a$\,. 
An equivalence class of linear hypercomplex structures is called a \emph{linear quaternionic structure}.\\   
\indent 
A \emph{hypercomplex/quaternionic vector space} is a vector space endowed with a linear hypercomplex/quaternionic structure. 
\end{defn}  

\indent 
Let $(E,\s)$ and $(E',\s')$ be hypercomplex vector spaces. A \emph{hypercomplex linear map} $\a:(E,\s)\to(E',\s')$ is a linear map such that 
$\a\bigl(\s(A)u\bigr)=\s'(A)\a(u)$\,, for any $A\in\mathfrak{gl}(2,\C\!)$ and $u\in E$\,.\\ 

\indent 
For the reader's convenience, we supply a proof for the following known result. 

\begin{prop} \label{prop:linear_hyper_classif}
Let $(E,\s)$ be a hypercomplex vector space. Then there exists a vector space $F$ and a hypercomplex linear isomorphism 
$\a:E\to\C^{\!2}\otimes F$, where the latter is endowed with the obvious linear hypercomplex structure.\\ 
\indent 
Moreover, for any such vector spaces $F$, $F'$ and hypercomplex linear isomorphisms $\a$, $\a'$ there exists a unique linear isomorphism 
$b:F\to F'$ such that $\a'=({\rm Id}_{\C^{\!2}}\otimes b)\circ\a$\,. 
\end{prop} 
\begin{proof} 
Note that, $\s$ is injective, because $\mathfrak{gl}(2,\C\!)$ is a simple associative algebra.\\ 
\indent  
Let $\t_k:\mathfrak{sl}(2,\C\!)\to {\rm End}\,U_k$ be the irreducible representation (of Lie algebras) with $\dim U_k=k+1$\,, $k\in\mathbb{N}$\,. 
As $\s$ is injective, the representation of $\mathfrak{sl}(2,\C\!)$ induced by $\s$ on $E$ decomposes into a direct sum of irreducible representations each of which 
is isomorphic to $\t_k$\,, for some $k$ in $M$ a finite subset of $\mathbb{N}$\,.\\ 
\indent 
Consequently, for each $k\in M$ we obtain a morphism of associative algebras 
from $\mathfrak{gl}(2,\C\!)$ to ${\rm End}\,U_k$ which when restricted to $\mathfrak{sl}(2,\C\!)$ gives $\t_k$\,, up to isomorphisms.\\ 
\indent 
It is easy to see that $0\notin M$, because any morphism of associative algebras from $\mathfrak{gl}(2,\C\!)$ to ${\rm End}\,U_0$ 
must have the kernel equal to $\mathfrak{sl}(2,\C\!)$ which is not an associative subalgebra of $\mathfrak{gl}(2,\C\!)$\,.\\ 
\indent 
Furthermore, as for any $A\in\mathfrak{sl}(2,\C\!)\setminus\{0\}$ with $A^2=0$ we have that $\t_k(A)$ is a nilpotent matrix of degree $k+1$\,, 
we deduce that $M=\{1\}$\,. The proof quickly follows.  
\end{proof} 

\indent 
Let $\bigl(E,[\s]\bigr)$ and $\bigl(E',[\s']\bigr)$ be quaternionic vector spaces. A \emph{quaternionic linear map} from $\bigl(E,[\s]\bigr)$ to 
$\bigl(E',[\s']\bigr)$ is a pair formed of a linear map $\a:E\to E'$ and an element $\pm a\in{\rm PGL}(2,\C\!)\,\bigl(={\rm GL}(2,\C\!)/\{\pm1\}\,\bigr)$ 
such that $\a\bigl(\s(A)u\bigr)=\bigl(\s'(A)\circ{\rm Ad}\,a\bigr)\a(u)$\,, for any $A\in\mathfrak{gl}(2,\C\!)$ and $u\in E$\,. 

\begin{cor} \label{cor:linear_q_classif} 
Let $\bigl(E,[\s]\bigr)$ be a quaternionic vector space. Then there exists a vector space $F$ and a quaternionic linear isomorphism 
$\a:E\to\C^{\!2}\otimes F$, where the latter is endowed with the obvious linear quaternionic structure.\\ 
\indent 
Moreover, for any such vector spaces $F$, $F'$ and quaternionic linear isomorphisms $\a$, $\a'$ there exists a pair  
$(a,b)$\,, with $a\in{\rm SL}(2,\C\!)$ and $b:F\to F'$ a linear isomorphism 
such that $\a'=(a\otimes b)\circ\a$\,; moreover, if $E$ is nontrivial, $(a,b)$ and $(-a,-b)$ are the only pairs with these properties.  
\end{cor}
\begin{proof} 
This is a quick consequence of Proposition \ref{prop:linear_hyper_classif}\,. 
\end{proof} 

\indent 
It is useful to reformulate Corollary \ref{cor:linear_q_classif}\,, as follows, where $\C\!P^1$ is embedded as the conic of nilpotent elements of 
$\mathfrak{sl}(2,\C\!)$\,.   

\begin{cor} \label{cor:exact_seq_linear_q_classif} 
Let $[\s]$ be a linear quaternionic structure on $E$\,. Then $\dim E=2k$\,, for some $k\in\mathbb{N}\setminus\{0\}$\,, and, on 
associating to any $A\in\mathfrak{sl}(2,\C\!)\setminus\{0\}$ with $A^2=0$ the kernel of $\s(A)$\,, we obtain a well defined (in particular, depending only of $[\s]$\,)
map $\zeta:\C\!P^1\to{\rm Gr}_k(E)$ with the following properties:\\ 
\indent 
\quad{\rm (i)} $\zeta$ is an embedding,\\  
\indent 
\quad{\rm (ii)} the pull back through $\zeta$ of the tautological exact sequence of vector bundles over ${\rm Gr}_k(E)$ gives 
\begin{equation} \label{e:exact_seq_linear_q_classif} 
0\longrightarrow\ol(-1)\otimes F\longrightarrow\ol\otimes E\longrightarrow\ol(1)\otimes F\longrightarrow0\;,  
\end{equation}
for some vector space $F$.\\  
\indent 
Conversely, any map from $\C\!P^1$ to ${\rm Gr}_k(E)$ satisfying {\rm (i)} and {\rm (ii)} is obtained this way from a unique linear quaternionic structure. 
\end{cor} 
\begin{proof} 
If $\zeta:\C\!P^1\to{\rm Gr}_k(E)$ satisfies (i) and (ii) then the cohomology sequence of \eqref{e:exact_seq_linear_q_classif} 
gives $E=H^0\bigl(\ol(1)\otimes F\bigr)=\C^{\!2}\otimes F$\,. 
\end{proof} 

\indent 
The correspondence of Corollary \ref{cor:exact_seq_linear_q_classif} is functorial in an obvious way.\\ 
\indent 
Let $\zeta:\C\!P^1\to{\rm Gr}_k(E)$ be the embedding corresponding to a linear quaternionic structure on $E$\,, $\dim E=2k$\,. 
We denote $E_z=\zeta(z)$\,, for any $z\in\C\!P^1$. 

\begin{defn}[\cite{Pan-qgfs}] \label{defn:linear_ro_q}
A \emph{linear $\r$-quaternionic structure} on a vector space $U$ is a pair $(E,\r)$\,, where $E$ is a quaternionic vector space 
and $\r:E\to U$ is a linear map such that $E_z\cap{\rm ker}\r=\{0\}$\,, for any $z\in\C\!P^1$. A \emph{$\r$-quaternionic vector space}  
is a vector space endowed with a linear $\r$-quaternionic structure. 
\end{defn} 

\indent 
If in Definition \ref{defn:linear_ro_q}\,, $E$ is a hypercomplex vector space then we obtain the notion of \emph{$\r$-hypercomplex vector space}.\\ 
\indent 
Let $(E,\r)$ be a linear $\r$-quaternionic structure on $U$. Then $z\mapsto\r(E_z)$\,, $(z\in\C\!P^1)$\,, defines an embedding of $\C\!P^1$ into ${\rm Gr}_k(U)$\,, 
where $\dim E=2k$\,. Moreover, the `restriction' to $\C\!P^1$ of the tautological vector bundle over ${\rm Gr}_k(U)$ is (isomorphic to) $\ol(-1)\otimes F$, 
for some vector space $F$, and we, thus, obtain an exact sequence 
\begin{equation} \label{e:exact_seq_linear_ro_q} 
0\longrightarrow\ol(-1)\otimes F\longrightarrow\ol\otimes U\longrightarrow\mathcal{U}\longrightarrow0\;; 
\end{equation}  
in particular, $\mathcal{U}$ is nonnegative.\\ 
\indent 
Note that \cite{Pan-qgfs}\,, the exact sequence \eqref{e:exact_seq_linear_ro_q} determines $(E,\r)$\,. Indeed, the cohomology sequence of its dual gives a linear 
map from $H^0(\ol\otimes U^*)=U^*$ to $H^0\bigl(k\ol(1)\bigr)=E^*$ which is the transpose of $\r$\,. 
We call $\mathcal{U}$ \emph{the (holomorphic) vector bundle of $(U,E,\r)$}\,.\\  
\indent 
Conversely, as \cite{Qui-QJM98} any nonnegative vector bundle over the sphere is uniquely obtained through an exact sequence like \eqref{e:exact_seq_linear_ro_q}\,, 
we obtain a functorial correspondence between $\r$-quaternionic vector spaces and nonnegative vector bundles over the sphere. 
Note that, here, the morphisms of vector bundles are covering (holomorphic) diffeomorphisms of the sphere. If we restrict to morphisms covering the identity 
map we obtain the basic functorial correspondence between $\r$-hypercomplex vector spaces and nonnegative vector bundles over the sphere.

\section{The infinitesimal Ward transform} \label{section:inf_Ward_transform} 

\indent 
We start this section with the following simple fact which will be useful later on. 

\begin{prop} \label{prop:useful_nonnegative} 
Let $\Phi:\mathcal{U}'\to\mathcal{U}$ be a surjective morphism of vector bundles over the sphere whose kernel is nonnegative.\\ 
\indent 
Then $\mathcal{U}$ is nonnegative if and only if $\mathcal{U}'$ is nonnegative; equivalently, $\mathcal{U}$ is the vector bundle of a $\r$-quaternionic vector space 
if and only if $\mathcal{U}'$ is the vector bundle of a $\r$-quaternionic vector space. 
\end{prop} 
\begin{proof} 
Let $\mathcal{U}''={\rm ker}\,\Phi$\,, and let $U$, $U'$ and $U''$ be the spaces of sections of $\mathcal{U}$\,, 
$\mathcal{U}'$, and $\mathcal{U}''$, respectively. The cohomology sequence of 
$0\longrightarrow\mathcal{U}''\longrightarrow\mathcal{U}'\overset{\Phi}{\longrightarrow}\mathcal{U}\longrightarrow0$ 
gives, because $\mathcal{U}''$ is nonnegative, an exact sequence of vector spaces 
$0\longrightarrow U''\longrightarrow U'\overset{\phi}{\longrightarrow}U\longrightarrow0$\,, 
where $\phi$ is a surjective linear map.\\ 
\indent 
We know that, for example, $\mathcal{U}$ is nonnegative if and only if, for any $z\in\C\!P^1$, the restriction map $r_z:U\to\mathcal{U}_z$\,, $s\mapsto s_z$\,, 
is surjective.\\ 
\indent 
Let $z\in\C\!P^1$, and denote by $r'_z$ and $r''_z$ the corresponding restriction maps of $\mathcal{U}'$ and $\mathcal{U}''$, respectively. 
As $\mathcal{U}''$ is nonnegative and the restriction of $r'_z$ to $U''$ is $r''_z$\,, we have $\mathcal{U}''_z\subseteq{\rm im}\,r'_z$\,.  
Hence, by Lemma \ref{lem:useful_surj}\,, below, $r'_z$ is surjective if and only if $\Phi\circ r'_z$ is surjective. 
On the other hand, as $\phi$ is surjective, $r_z$ is surjective if and only if $r_z\circ\phi$ is surjective. 
Together with $\Phi\circ r'_z=r_z\circ\phi$\,, this completes the proof. 
\end{proof} 

\begin{lem} \label{lem:useful_surj} 
Let $T:U\to V$ be a linear map and let $W\subseteq V$ be a vector subspace; denote by $p:V\to V/W$ the projection.\\ 
\indent  
Then $T$ is surjective if and only if $p\circ T$ is surjective and $W\subseteq{\rm im}\,T$. 
\end{lem} 
\begin{proof} 
Note that the condition $W\subseteq {\rm im}\,T$ is equivalent to the condition that the annihilator of $W$ contains the kernel 
of the transpose of $T$. Thus, the dual of the conclusion reads: $T$ is injective if and only if $T|_W$ is injective and 
${\rm ker}\,T\subseteq W$. As this is obvious, the proof is complete. 
\end{proof}  
  
\begin{prop} \label{prop:equiv_trivial-q_isom}
Let $(U,E,\r)$ and $(U',E',\r')$ be $\r$-quaternionic vector spaces and let $\mathcal{U}$ and $\mathcal{U}'$ be their vector bundles, respectively.\\ 
\indent 
Let $\phi:U'\to U$ be a surjective $\r$-quaternionic linear map. Denote by $\widetilde{\phi}:E'\to E$ the quaternionic linear map, and by 
$\Phi:\mathcal{U}'\to\mathcal{U}$ the (surjective) morphism of vector bundles corresponding to $\phi$\,.\\ 
\indent 
Then the following assertions are equivalent:\\ 
\indent 
\quad{\rm (i)} $\widetilde{\phi}$ is an isomorphism,\\ 
\indent 
\quad{\rm (ii)} $\dim E=\dim E'$ and $E'_z\cap{\rm ker}\,\phi=\{0\}$\,, for any $z\in\C\!P^1$.\\  
\indent 
\quad{\rm (iii)} ${\rm ker}\,\Phi$ is a trivial vector bundle.\\ 
\indent 
Moreover, if {\rm (i)}\,, {\rm (ii)} or {\rm (iii)} hold then ${\rm ker}\,\Phi=\ol\otimes{\rm ker}\,\phi$\,.  
\end{prop} 
\begin{proof} 
Firstly, note that both (i) and (ii) are equivalent to the fact that (up to a diffeomorphism of the sphere) $\phi$ maps each $E'_z$ isomorphically onto $E_z$\,.\\  
\indent 
Now, note that, ${\rm ker}\,\phi$ is the space of sections of ${\rm ker}\,\Phi$\,. Thus,  $E'_z\cap{\rm ker}\,\phi=\{0\}$\,, for any $z\in\C\!P^1$, 
if and only if the following fact holds: (a) the nonzero sections of ${\rm ker}\,\Phi$ are nonzero at each point.\\ 
\indent 
Also, $\dim E=\dim E'$ is equivalent to the following fact: (b) the dimension of ${\rm ker}\,\phi$ is equal to the rank of ${\rm ker}\,\Phi$\,.\\ 
\indent 
To complete the proof of the equivalence of (i)\,, (ii) and (iii)\,, just note that (iii) holds if and only if both (a) and (b) hold. 
\end{proof} 

\indent 
Note that, if $(E,\r)$ is a linear $\r$-quaternionic structure on $U$ then $\r:E\to U$ is a $\r$-quaternionic linear map. 
Thus, if $\mathcal{E}$ and $\mathcal{U}$ are the vector bundles of $E$ and $U$, respectively, then $\r$ corresponds to 
a morphism of vector bundles $\mathcal{R}:\mathcal{E}\to\mathcal{U}$ whose kernel and cokernel are $\ol\otimes{\rm ker}\r$ and $\ol\otimes{\rm coker}\r$\,, 
respectively. Consequently, the following assertions are equivalent:\\ 
\indent 
\quad(i) $\r$ is surjective,\\  
\indent 
\quad(ii) $\mathcal{R}$ is surjective,\\ 
\indent 
\quad(iii) $\mathcal{U}$ is positive.\\
Thus, if $\r$ is surjective then 
$0\longrightarrow\ol\otimes{\rm ker}\r\longrightarrow\mathcal{E}\overset{\mathcal{R}}{\longrightarrow}\mathcal{U}\longrightarrow0$\,,  
and the cohomology sequence of the dual of this exact sequence gives $({\rm ker}\r)^*=H^1\bigl(\mathcal{U}^*\bigr)$\,.\\ 
\indent 
The first part of the next result can be seen as the infinitesimal Ward transform.  

\begin{thm} \label{thm:infinitesimal_Ward_transform} 
Let $0\longrightarrow\mathcal{U}''\longrightarrow\mathcal{U}'\overset{\Phi}{\longrightarrow}\mathcal{U}\longrightarrow0$ 
be an exact sequence of holomorphic vector bundles over the sphere with $\mathcal{U}$ nonnegative and $\mathcal{U}''$ trivial.\\  
\indent 
Then $\mathcal{U}'$ is nonnegative and, if $(U,E,\r)$ and $(U',E',\r')$ are the $\r$-quaternionic vector spaces corresponding 
to $\mathcal{U}$ and $\mathcal{U}'$, respectively, then the $\r$-quaternionic linear map $\phi:U\to U'$ determined by $\Phi$ 
is surjective and induced by a quaternionic linear isomorphism $E=E'$\,; moreover, any such $\r$-quaternionic linear map is 
obtained this way from a surjective morphism of nonnegative vector bundles whose kernel is trivial.\\ 
\indent 
Furthermore, if $\mathcal{U}$ is positive then, under the linear isomorphisms $E=E'$ and $({\rm ker}\r)^*=H^1\bigl(\mathcal{U}^*\bigr)$\,, 
the linear map $\r'|_{{\rm ker}\r}$ corresponds to the opposite of the cohomology class of the given exact sequence;  
also, $U'$ is a quaternionic vector space (equal to $E'$) if and only if $\r'|_{{\rm ker}\r}$ is an isomorphism. 
\end{thm} 
\begin{proof} 
Proposition \ref{prop:useful_nonnegative} gives that $\mathcal{U}'$ is nonnegative, whilst Proposition \ref{prop:equiv_trivial-q_isom} 
gives the identification $E=E'$ and the fact that $\mathcal{U}''=\ol\otimes{\rm ker}\,\phi$\,.\\ 
\indent 
As (under $E=E'$) we have $\phi\circ\r'=\r$ we, also, have $\r'({\rm ker}\r)\subseteq{\rm ker}\,\phi$\,.  
Thus, if $\mathcal{U}$ is positive and we denote by $\mathcal{E}$ the vector bundle of $E$\,, we obtain the following commutative diagram 

\begin{displaymath} 
\xymatrix{
   0  \ar[r]   &  \ol\otimes{\rm ker}\r   \ar[r] \ar[d]    &  \mathcal{E} \ar[r] \ar[d]        &   \mathcal{U} \ar[r] \ar[d]^=    &       0  \\
   0  \ar[r]   &  \ol\otimes{\rm ker}\,\phi \ar[r]          &  \mathcal{U}' \ar[r]^{\Phi}     &   \mathcal{U} \ar[r]                  &   0 \;,     }
\end{displaymath} 

\noindent 
which, through the cohomology sequences of the rows of its dual, gives 

\begin{displaymath} 
\xymatrix{
          &          0                                             &  H^1\bigl(\mathcal{U}^*\bigr) \ar[l]                  &    ({\rm ker}\r)^*    \ar[l]_{\;\;\simeq}     &       0  \ar[l]      &  \\
   0    &    H^1\bigl(\mathcal{U}'^*\bigr) \ar[l] \ar[u] &  H^1\bigl(\mathcal{U}^*\bigr) \ar[l]  \ar[u]^=   &   ({\rm ker}\,\phi)^*  \ar[l]_{\;\;o}  \ar[u] 
   &  H^0\bigl(\mathcal{U}'^*\bigr) \ar[l]  \ar[u] &   0 \;, \ar[l]   }
\end{displaymath} 

\noindent 
where, because $\ol\otimes({\rm ker}\,\phi)^*$ is trivial, the linear map $o:({\rm ker}\,\phi)^*\to H^1(\mathcal{U}^*)$ is,  
up to the equality of ${\rm Hom}\bigl(({\rm ker}\,\phi)^*, H^1\bigl(\mathcal{U}^*\bigr)\bigr)$ 
and $H^1\bigl({\rm Hom}\bigl(\ol\otimes({\rm ker}\,\phi)^*, \mathcal{U}^*\bigr)\bigr)$\,, the  
cohomology class of the dual of the given exact sequence (determined by $\Phi$\,). As this cohomology class is the opposite (of the transpose) 
of the cohomology class of the given sequence \cite[Proposition 3]{At-57}\,, it remains to prove the last statement.  
But $o$ is an isomorphism if and only if the cohomology groups of the dual of $\mathcal{U}'$ are zero. As the latter condition is equivalent 
to $\mathcal{U}'=\ol(1)\otimes F$, for some vector space $F$, the proof is complete. 
\end{proof} 

\indent 
As a short exact sequence of vector bundles is determined, only up to an equivalence class, by its cohomology class (see \cite{At-57}\,), it is, obviously, 
useful to work with cocycles, instead. This way, we obtain another description of the isomorphism 
$({\rm ker}\r)^*=H^1\bigl(\mathcal{U}^*\bigr)$, where $(U,E,\r)$ is a $\r$-quaternionic vector space, with $\r$ surjective, and $\mathcal{U}$ 
is its vector bundle. 

\begin{rem} \label{rem:cohomology_for_pluriharmonic} 
Let $\mathcal{U}$ be a positive vector bundle over the sphere and let $(U,E,\r)$ be the corresponding $\r$-quaternionic vector space. 
Then $U=H^0(\mathcal{U})$ and, on fixing a linear quaternionic isomorphism $E=\C^{\!2}\otimes F$, the dual of \eqref{e:exact_seq_linear_ro_q} gives 
\begin{equation} \label{e:cohomology_for_pluriharmonic} 
0\longrightarrow\mathcal{U}^*\overset{\iota}{\longrightarrow}\ol\otimes U^*\overset{\t}{\longrightarrow}\ol(1)\otimes F^*\longrightarrow0\;.  
\end{equation}   
\indent 
Identify the sphere with $\C\sqcup\{\infty\}$ and take a cover of it $\{D_0,D_{\infty}\}$ such that $D_0$ is an open disk with centre $0$ 
containing the unit circle $S^1$ (with centre $0$), and $D_{\infty}$ is the complement of a closed disk with centre $0$ and which is disjoint from $S^1$; 
in particular, $D_0\cap D_{\infty}$ is an open annulus containing $S^1$.\\ 
\indent 
Any element of $H^1\bigl(\mathcal{U}^*\bigr)$ is represented by a section $h$ of $\mathcal{U}^*$ 
over $D_0\cap D_{\infty}$\,. Then there exist unique holomorphic maps $h_0:D_0\to U^*$ and $h_{\infty}:D_{\infty}\to U^*$ 
such that $h_{\infty}(\infty)=0$ and $\iota\circ h=h_0-h_{\infty}$ on $D_0\cap D_{\infty}$\,. 
By using the Cauchy formula we deduce that, up to a constant factor, we have  
$$\int_0^{2\p}(\iota\circ h)(\g(t))\dif\!t=h_0(0)\;,$$ 
where $\g$ is the canonical parametrization of $S^1$.\\ 
\indent 
Now, there is a unique $\xi\in E^*$ whose corresponding section $s_{\xi}$ of $\ol(1)\otimes F^*$ is equal to $\t\circ h_0$ at $0$ and is zero at $\infty$\,.  
Equivalently, as $h_0$ and $h_{\infty}$ differ by a section of ${\rm ker}\,\t$, on the intersection of their domains, $\t\circ h_0$ and $\t\circ h_{\infty}$ 
are the restrictions of $s_{\xi}$ to $D_0$ and $D_{\infty}$\,, respectively.\\ 
\indent 
The space of sections of $\ol(1)\otimes F^*$ which are zero at $\infty$ is supplementary to the image of $H^0(\t)$\,. 
Consequently, $h\mapsto\a(h)=\int_0^{2\p}(\iota\circ h)(\g(t))\dif\!t$ is a natural (up to the choice of the `Poles') lift of the isomorphism 
$H^1\bigl(\mathcal{U}^*\bigr)=E^*/U^*\bigl(=({\rm ker}\r)^*\bigr)$ 
determined by the cohomology sequence of \eqref{e:cohomology_for_pluriharmonic}\,. 
Finally, note that, the element of $({\rm ker}\r)^*$ corresponding to the cohomology class $[h]$ is 
the image of $\a(h)$ through a surjective linear map $p:U^*\to({\rm ker}\r)^*$, whose kernel 
is $\bigl(U^*\cap{\rm Ann}(E_0)\bigr)+\bigl(U^*\cap{\rm Ann}(E_{\infty})\bigr)$\,, where $U^*$ is embedded into $E^*$ through $\r^*$. 
\end{rem}

\section{The Ward transform} \label{section:Ward_transform} 

\indent 
A \emph{bundle of associative algebras} is a vector bundle whose typical fibre is an associative algebra and whose structural group is the automorphism group of that algebra.\\ 
\indent   
A \emph{quaternionic vector bundle} is a vector bundle $E$ endowed with a pair $(\mathcal{A},\s)$\,, where $\mathcal{A}$ is a bundle of associative algebras, 
with typical fibre ${\rm gl}(2,\C\!)$\,, and $\s:\mathcal{A}\to{\rm End}E$ is a morphism of bundles of associative algebras. 
If $\mathcal{A}=M\times{\rm gl}(2,\C\!)$ then $E$ is a \emph{hypercomplex vector bundle}.\\ 
\indent 
Alternatively, a quaternionic vector bundle is a vector bundle $E$ with typical fibre $\C^{\!2}\otimes\C^{\!k}$ and structural group 
${\rm SL}(2,\C\!)\cdot{\rm GL}(k,\C\!)$ acting on the typical fibre such that to any $\pm(a,A)\in{\rm SL}(2,\C\!)\cdot{\rm GL}(k,\C\!)$ 
we associate $a\otimes A$\,, $(k\in\mathbb{N})$\,.\\ 
\indent  
A hypercomplex vector bundle can be defined, also, as a vector bundle $E$ with typical fibre $\C^{\!2}\otimes\C^{\!k}$  
and structural group ${\rm GL}(k,\C\!)$ acting on the typical fibre such that to any $A\in{\rm GL}(k,\C\!)$ 
we associate ${\rm Id}_{\C^{\!2}}\otimes A$\,, $(k\in\mathbb{N})$\,.\\ 
\indent 
Let $(E,\mathcal{A},\s)$ be a quaternionic vector bundle. The kernels of the nilpotent elements of $\mathcal{A}\setminus0$ form the total space of a sphere bundle 
$\p:Y\to M$. Moreover, on associating to any nilpotent $A\in\mathcal{A}\setminus0$ the kernel of $\s(A)$ we obtain a bundle morphism 
$\zeta:Y\to{\rm Gr}_k(E)$ which determine $(\mathcal{A},\s)$\,, where $\rank E=2k$\,. With respect to the principal bundle 
with structural group ${\rm SL}(2,\C\!)\cdot{\rm GL}(k,\C\!)$\,, to which $E$ is associated, $Y$ is associated through the morphism of Lie groups 
${\rm SL}(2,\C\!)\cdot{\rm GL}(k,\C\!)\to{\rm PGL}(2,\C\!)$\,, $\pm(a,A)\mapsto\pm a$\,. 

\begin{defn} 
A \emph{$\r$-quaternionic ($\r$-hypercomplex) vector bundle} is a vector bundle $U$, over a manifold $M$, endowed with a pair $(E,\r)$\,, 
where $E$ is a quaternionic (hypercomplex) vector bundle over $M$ and $\r:E\to U$ is a morphism of vector bundles such that $(E_x,\r_x)$ 
is a linear $\r$-quaternionic ($\r$-hypercomplex) structure on $U_x$\,, for any $x\in M$.\\ 
\indent 
An  \emph{almost $\r$-quaternionic ($\r$-hypercomplex) manifold} is a manifold whose tangent bundle is $\r$-quaternionic ($\r$-hypercomplex). 
\end{defn} 

\indent 
Let $(M,E,\r)$ be an almost $\r$-quaternionic manifold, $\rank E=2k$\,. Denote $H={\rm SL}(2,\C\!)\cdot{\rm GL}(k,\C\!)$ and let $(P,M,H)$ 
be the bundle of `quaternionic frames' of $E$\,; that is, $E$ is associated to $P$ through the action of $H$ on $\C^{\!2}\otimes\C^{\!k}$.\\ 
\indent  
To define the notion of integrability, adequate to this context, we need the `principal $\r$-connections' of \cite{Pan-qgfs}\,. 
For $(P,M,H)$ and $\r:E\to M$ this is as follows. Let $TP/H$ be the vector bundle over $M$ whose sections, over some 
open subset $U\subseteq M$, are the $H$-invariant vector fields on $P|_U$ \cite{At-57}\,. Denote by $\widetilde{\dif\!\p}_P$ 
the morphism of vector bundles from $TP/H$ to $TM$ induced by $\dif\!\p_P$\,, where $\p_P:P\to M$ is the projection.\\ 
\indent  
A \emph{principal $\r$-connection} on $P$ is a morphism of vector bundles $c_P:E\to TP/H$ such that $\widetilde{\dif\!\p}_P\circ c_P=\r$\,. 
Then, similarly to the case of (classical) principal connections \cite{At-57}\,, 
the obstruction to the existence of principal $\r$-connections on $P$ is an element of $H^1\bigl(M,{\rm Hom}(E,{\rm Ad}P)\bigr)$\,. 
Furthermore, up to the curvature (for which, one needs a suitable bracket on the sheaf of sections of $E$) the theory of principal connections 
admits a straight generalization to principal $\r$-connections. In particular, any principal $\r$-connection $c_P$ on $P$ determines 
an \emph{associated connection} on $Y$ which is a morphism of vector bundles $c:\p^*E\to TY$ which when composed with the morphism 
from $TY$ to $\p^*(TM)$\,, determined by $\dif\!\p$\,, gives $\p^*\r$\,. Now, let $\Cal$ be the distribution on $Y$ 
given by $\Cal_y=c\bigl(\zeta(y)\bigr)$\,, for any $y\in Y$. 

\begin{defn}[\,\cite{Pan-qgfs}\,]  
The almost $\r$-quaternionic structure $(E,\r)$ is \emph{integrable, with respect to $c$}\,, if $\Cal$ is integrable. 
Then $(M,E,\r,c)$ is a \emph{$\r$-quaternionic manifold}.\\ 
\indent 
Suppose, further, that there exists a surjective submersion $\psi:Y\to Z$ such that ${\rm ker}\dif\!\psi=\Cal$\,, and $\psi$ restricted to each fibre of $\p$ 
is injective. Then $Z$ is the \emph{twistor space} of $(M,E,\r,c)$\,, and, for any $x\in M$, the sphere $\psi(\p^{-1}(x))$\,, embedded into $Z$\,, 
is a \emph{twistor sphere}. 
\end{defn}  

\indent 
Let $(M,E,\r,c)$ be a $\r$-quaternionic manifold such that $E$ is a hypercomplex vector bundle and $c$ is (induced by) the trivial flat 
connection corresponding to the isomorphism $Y=M\times\C\!P^1$. Then $(M,E,\r,c)$ is a \emph{$\r$-hypercomplex manifold}.\\ 
\indent 
Let $(M,E,\r,c)$ be a $\r$-hypercomplex manifold with twistor space $Z$ given by the surjective submersion $\psi:Y\to Z$. 
We shall, further, assume that the projection $Y=M\times\C\!P^1\to\C\!P^1$ factorises into $\psi$ followed by a submersion from $Z$ onto $\C\!P^1$.\\ 
\indent 
A $\r$-quaternionic manifold for which $\r$ is surjective is called a \emph{quaternionic object}. A $\r$-hypercomplex manifold for which $\r$ is surjective 
is called a \emph{hypercomplex object}. 

\begin{defn}  \label{defn:asd_connection} 
Let $(M,E,\r,c)$ be a $\r$-quaternionic manifold; denote by $\p:Y\to M$ the sphere bundle induced by $(E,\r)$ and by $\Cal$ the foliation induced by $c$.\\ 
\indent 
Let $(P,M,G)$ be a principal bundle endowed with a principal $\r$-connection $c_P$. We say that $c_P$ is \emph{anti-self-dual} if (locally) its restriction   
to the image through $\p$ of any leaf of $\Cal$ is (induced by) a flat principal connection. 
\end{defn}  

\indent 
It is useful to have an alternative description of the anti-self-duality condition for connections. For this, with the same notations as in Definition \ref{defn:asd_connection}\,, 
note that, $c$ allows us to take the pull back through $\p$ of any principal $\r$-connection $c_P$ on $(P,M,G)$\,. 
Then $c_P$ is anti-self-dual if and only if $\p^*(c_P)$ restricted to $\Cal$ is flat.\\ 
\indent 
Now, we can give the main result of this section. 

\begin{thm} \label{thm:Ward_transform} 
Let $(M,E,\r,c)$ be a $\r$-quaternionic manifold with twistor space $Z$ given by $\psi:Y\to Z$.\\ 
\indent 
Then for any principal bundle $(\mathcal{P},Z,G)$\,, whose restriction to each twistor sphere is trivial, 
there exists a unique principal bundle $(P,M,G)$\,, endowed with an anti-self-dual principal $\r$-connection, such that 
$\psi^*\mathcal{P}=\p^*P$.\\ 
\indent 
Conversely, if the fibres of $\psi$ are simply connected then any principal bundle over $M$, endowed with an anti-self-dual principal $\r$-connection, 
is obtained, this way, from a unique principal bundle over $Z$, whose restriction to each twistor sphere is trivial. 
\end{thm} 
\begin{proof} 
Let $(\mathcal{P},M,G)$ be a principal bundle over $Z$ whose restriction to each twistor sphere is trivial; denote by $\Phi$ its projection. 
Then through each point of $\mathcal{P}$ passes an embedded sphere which is diffeomorphically projected by $\Phi$ onto a twistor sphere. Moreover, 
if $t\subseteq\mathcal{P}$ is such a sphere then $\dif\!\Phi$ induces an exact sequence 
\begin{equation} \label{e:twistor_spheres_in_P_and_M} 
0\longrightarrow{\rm ker}(\dif\!\Phi|_t)\longrightarrow\mathcal{N}t\longrightarrow Nt\longrightarrow0\;, 
\end{equation} 
where $\mathcal{N}t$ and $Nt$ are the normal bundles of $t$ and $\Phi(t)$ into $\mathcal{P}$ and $Z$, respectively. 
From Theorem \ref{thm:infinitesimal_Ward_transform} we obtain that $\mathcal{N}t$ is nonnegative. Hence, by \cite{Pan-qgfs}\,, 
there exists a $\r$-quaternionic manifold $(P,E',\r',c')$ whose twistor space is $\mathcal{P}$. Moreover, $G$ acts freely on $P$ and 
$\Phi$ corresponds to a surjective submersion $\phi:P\to M$\,; consequently, $(P,M,G)$ is a principal bundle. 
Also, the sphere bundle of $(E',\r')$ can be canonically identified with both $\psi^*\mathcal{P}$ and $\p^*P$.\\ 
\indent 
Now, by applying, again, Theorem \ref{thm:infinitesimal_Ward_transform}\,, we deduce that $E'=\phi^*E$ and $\widetilde{\dif\!\phi}\circ\r'=\phi^*\r$\,, 
where $\widetilde{\dif\!\phi}:TP\to\phi^*(TM)$ is the morphism of vector bundles induced by $\dif\!\phi$\,. 
Thus, as $\r'$ is also $G$-invariant, it defines a principal $\r$-connection on $(P,M,G)$\,, 
whose pull back by $\p$ (with respect to $c$) must be flat along the fibres of $\psi$\,, because $\psi^*\mathcal{P}=\p^*P$.\\ 
\indent 
Conversely, suppose that the fibres of $\psi$ are simply-connected and let $(P,M,G)$ be a principal bundle endowed with an anti-self-dual connection $c_P$. 
Then there exists a principal bundle $(\mathcal{P},Z,G)$ such that $\psi^*\mathcal{P}=\p^*P$. As the restriction of $\p^*P$ to each fibre of $\p$ is trivial, 
the proof is complete. 
\end{proof}

\section{The Ward transform and hyper-K\"ahler manifolds} \label{section:WhK} 

\indent 
Let $(M,E,\r,c)$ be a hypercomplex object with twistor space $Z$ given by the surjective submersion $\psi:Y=M\times\C\!P^1\to Z$. 
Recall that the normal bundles of the twistor spheres are positive, and we assume that the projection $Y\to\C\!P^1$ factorises as $\chi\circ\psi$\,, 
where $\chi:Z\to\C\!P^1$ is a surjective submersion; in particular, the restriction of $\chi$ to each twistor sphere is a diffeomorphism. 
Therefore, for any $z\in\C\!P^1$, the restriction of $\psi$ to $M\times\{z\}$ defines a surjective submersion $\psi_z:M\to\chi^{-1}(z)$\,.\\ 
\indent 
Let $(\mathcal{P},Z,G)$ be a principal bundle whose restriction to each twistor sphere is trivial, and let $(P,M,G)$ be the principal bundle 
endowed with a principal $\r$-connection $c$ obtained through the Ward transform. Recall that $\r:E=\C^{\!2}\otimes F\to TM$ is the 
surjective morphism of vector bundles defining the almost $\r$-hypercomplex structure of $M$, where $F$ is some vector bundle over $M$. 
Also, $c:E\to TP/G$ is a morphism of vector bundles such that $\widetilde{\dif\!\phi}\circ c=\r$\,, 
where $\phi:P\to M$ is the projection and $\widetilde{\dif\!\phi}:TP/G\to TM$ 
is the induced morphism of vector bundles. In particular, the restriction of $c$ to ${\rm ker}\r$ takes values into ${\rm Ad}P\,(\subseteq TP/G)$\,.\\ 
\indent 
From Theorem \ref{thm:infinitesimal_Ward_transform}\,, we obtain that 
\emph{$c|_{{\rm ker}\r}:{\rm ker}\r\to{\rm Ad}P$ is an isomorphism if and only if $P$ is hypercomplex and its twistor space is $\mathcal{P}$}.\\ 
\indent 
From now on, in this section, we chall assume that $c|_{{\rm ker}\r}$ is an isomorphism, and we shall seek 
sufficient conditions under which $P$ is endowed with a compatible hyper-K\"ahler structure. 
Let $\mathcal{E}\subseteq T\mathcal{P}/G$ be the preimage of ${\rm ker}\chi$ through $\widetilde{\dif\!\Phi}$\,, 
where $\Phi:\mathcal{P}\to Z$ is the projection and $\widetilde{\dif\!\Phi}:T\mathcal{P}/G\to TZ$ is induced by $\dif\!\Phi$\,. 
Then the restriction of $\Fa=\chi^*\bigl(\ol(-1)\bigr)\otimes\mathcal{E}$ to each twistor sphere is trivial;  
moreover, the bundle over $M$ induced by the Ward transform is $F$. (Note that, the condition $\Fa$ restricted to each twistor sphere be trivial 
is equivalent to $c|_{{\rm ker}\r}$ be an isomorphism.) 

\begin{prop} \label{prop:hK} 
Any compatible hyper-K\"ahler metric on $P$, invariant under $G$, corresponds to a linear symplectic structure on $\Fa$.\\ 
\indent  
Moreover, if $\dim P\geq8$ then any linear conformal symplectic structure on $\Fa$ induces, locally, a compatible hyper-K\"ahler metric on $P$. 
\end{prop}  
\begin{proof} 
Let $\nabla$ be the Obata connection on $P$. This is characterized by the fact that it is torsion free and compatible with the 
almost hypercomplex structure of $P$. It can be, also, described as follows. Let $\chi_P=\chi\circ\Phi$\,, where $\Phi:\mathcal{P}\to Z$ 
is the projection, and let $E_P=\phi^*E=\C^{\!2}\otimes F_P$\,, where $F_P=\phi^*F$. Then $F_P$ corresponds, through the Ward transform,  
to $\Fa_P=\chi_P^*\bigl(\ol(-1)\bigr)\otimes({\rm ker}\dif\!\chi_P)$\,; denote by $\nabla^{F_P}$ the corresponding anti-self-dual connection on $F_P$\,. 
Then $\nabla$ is the tensor product of the trivial connection on $P\times\C^{\!2}$ and $\nabla^{F_P}$.\\ 
\indent 
Now, any compatible hyper-K\"ahler metric on $P$ corresponds to a linear symplectic structure on $F_P$ 
which is covariantly constant with respect to $\nabla^{F_P}$. Together with the Ward transform, this shows that any 
compatible hyper-K\"ahler metric on $P$ corresponds to a linear symplectic structure on $\Fa_P$\,. This quickly implies the first statement.\\ 
\indent 
For the second statement, note that, any linear conformal symplectic structure on $\Fa_P$ induces a conformal structure on $P$ which 
is Hermitian; that is, for any $z\in\C\!P^1$, the fibres of $\psi_z$ are isotropic with respect to this conformal structure. 
As $\nabla$ preserves this conformal structure and $\dim P\geq8$\,, this completes the proof. 
\end{proof}  

\indent 
We know that the $\r$-quaternionic vector spaces whose vector bundles are positive line bundles are given by the 
irreducible representations $U_k$ of ${\rm SL}(2,\C\!)$\,, with $\dim U_k=k+1$\,, $k\in\mathbb{N}\setminus\{0\}$\,. 
The linear $\r$-quaternionic structure of $U_k$ is given by $E=U_1\otimes U_{k-1}=U_{k-2}\oplus U_k$ and the projection 
$\r:E\to U_k$\,, $(k\in\mathbb{N}\setminus\{0\})$\,, where $U_0$ is the trivial one-dimensional representation. 
Also, the vector bundle of $U_k$ is $\ol(k)$\,, and, if $k$ is even, there exists a ${\rm SL}(2,\C\!)$-invariant Euclidean structure 
$h_k$ on $U_k$\,, unique up to homotheties.    

\begin{thm} \label{thm:hK_Z_surface}
If $\dim Z=2$ and $4$ divides $\dim P\geq8$ then $P$ is endowed with a compatible hyper-K\"ahler metric. 
\end{thm} 
\begin{proof} 
Let $k\geq4$ be such that $\dim P=2k$\,. Then the normal bundle of each twistor sphere in $Z$ is $\ol(k)$\,, 
and we claim that the frame bundle of $\Fa$ admits a reduction to ${\rm GL}(2,\C\!)\,\bigl(={\rm GL}(U_1)\bigr)$ 
through its representation on $U_{k-1}\,(=\odot^{k-1}U_1)$\,, restricting to the corresponding irreducible representation of ${\rm SL}(2,\C\!)$ .\\ 
\indent 
Indeed, let $z\in Z$ and let $t\subseteq Z$ be a twistor sphere such that $z\in t$\,. Then $\Fa|_t$ is trivial 
and $\widetilde{\dif\!\Phi}$ induces a surjective morphism of vectors bundles $\mu_t:\mathcal{F}|_t\to\chi_P^*\bigl(\ol(k-1)\bigr)$  
such that $H^0(\mu_t)$ is an isomorphism. Consequently, the fibre $\Fa_z$ is endowed with a linear $\r$-quaternionic structure, 
whose vector bundle is $\ol(k-1)$ and is therefore determined, up to a nonzero factor, by a Veronese embedding $\zeta_t:t\to P\mathcal{F}_z^*$\,. 
Now, for the generic twistor sphere $t'\subseteq Z$ with $z\in t'$ we have that $t$ and $t'$ intersect in $k$ `linearly independent' points 
(because the twistor spheres in $Z$ have normal bundle $\ol(k)$\,). But $\zeta_t$ and $\zeta_{t'}$ are equal at each point of $t\cap t'$\,. 
Hence, up to the diffeomorphism $t=t'$, induced by $\chi_P$\,, we have $\zeta_t=\zeta_{t'}$, and the claim is proved.\\ 
\indent 
As $k$ is even, any nonzero element of $\Lambda^2U_1$ determines a linear conformal symplectic structure on $\Fa$ and, 
by Proposition \ref{prop:hK}\,, the proof is complete. 
\end{proof} 

\indent 
Next, in this section, we concentrate on proving that $\phi$ of Theorem \ref{thm:hK_Z_surface} is a harmonic morphism. 
Firstly, note that, if $Z$ is a surface and $\dim M=k+1$ then $\phi$ is horizontally conformal with respect 
to the conformal structure induced by $h_k$ on $M$. Indeed, up to nonzero factors, at each point of $M$, the differential of $\phi$ 
is modelled by the projection $U_1\otimes U_{k-1}=U_{k-2}\oplus U_k\to U_k$\,. Now, we need the following. 

\subsection*{The Obata $\r$-connection} 

\indent 
We shall use the same notations as in the beginning of this section with the further assumption that the normal bundle in $Z$ 
of one (and, hence, any) twistor sphere is $l\ol(k)$ for some positive integers $k$ and $l$\,. Therefore the restriction 
of ${\rm ker}\dif\!\chi$ to each twistor sphere is isomorphic to $l\ol(k)$\,. Thus, on denoting 
$\Fa_M=\chi^*\bigl(\ol(-k)\bigr)\otimes({\rm ker}\dif\!\chi)$\,, then the restriction of $\Fa_M$ to each twistor sphere is trivial. 
Let $F_M$ be the vector bundle over $M$ endowed with an anti-self-dual $\r$-connection $\nabla^{F_M}$ determined by $\Fa_M$\,, 
through the Ward transform; obviously, $\rank F_M=l$\,.\\ 
\indent 
Consequently, $TM=(M\times U_k)\otimes F_M$ and we call the the tensor product $\nabla^M$ 
of (the $\r$-connection induced by) the trivial connection on $M\times U_k$ and $\nabla^{F_M}$ 
\emph{the Obata $\r$-connection of $M$}.\\ 
\indent 
To characterise the Obata $\r$-connection, note that, up to integrability, the structure of $M$ is given by a reduction of its frame bundle  
to ${\rm GL}(l,\C\!)$ through its representation on $U_k\otimes\C^{\!l}$, given by $a\mapsto{\rm Id}_{U_k}\otimes a$\,, for any $a\in{\rm GL}(l,\C\!)$\,. 

\begin{prop} \label{prop:charact_Obata} 
The Obata $\r$-connection $\nabla^M$ is characterized by the following:\\ 
\indent 
{\rm (1)} $\nabla^M$ is compatible with the almost ${\rm GL}(l,\C\!)$-structure of $M$\,; in particular, $\nabla^M$ preserves ${\rm ker}\dif\!\psi_z$\,, 
for any $z\in\C\!P^1$.\\ 
\indent 
{\rm (2)} The $\r$-connection induced by $\nabla^M$ on the normal bundle of each fibre of $\psi_z$ is given by the partial Bott connection defined by 
${\rm ker}\dif\!\psi_z$\,, for any $z\in\C\!P^1$.  
\end{prop}  
\begin{proof} 
Condition (1) applied to any connection $\nabla$ is equivalent to the fact that $\nabla$ is the tensor product of the trivial flat connection on $M\times U_k$ 
and a $\r$-connection $\nabla_1$ on $F_M$\,. Assuming that this holds, $\nabla_1$ is anti-self-dual if and only if, for any $z\in\C\!P^1$, the restriction 
of $\nabla_1$ to the fibres of $\psi_z$ is flat.\\ 
\indent 
Let $z\in\C\!P^1$. It is easy to see that the restriction of $F_M$ to any fibre of $\psi_z$ is isomorphic to the normal bundle of that fibre. 
Therefore, up to this isomorphisms, condition (2) applied to $\nabla_1$ and $z$ says that the restriction of $\nabla_1$ to each fibre of 
$\psi_z$ is the flat connection whose covariantly constant sections are characterized by the fact that are projectable with respect to $\psi_z$\,.\\ 
\indent 
Now, recall that $F_M$ is characterised by $\psi^*\bigl(\Fa_M\bigr)=\p^*\bigl(F_M)$\,, and $\nabla^{F_M}$ is characterized by the fact that its pull back to 
$Y=M\times\C\!P^1$ is the trivial flat connection when restricted to the fibres of $\psi$\,. As, essentially, $\psi_z$ is the restriction of $\psi$ to $M\times\{z\}$   
we deduce that condition (2) characterizes $\nabla^M$ among the $\r$-connections satisfying (1)\,. 
\end{proof} 

\indent 
Suppose, now, that $l=1$ and $k$ is even, and let $[h]$ be the conformal structure on $M$ induced by $h_k$\,. The almost $\r$-hypercomplex structure 
of $M$ is given by $(E,\r)$\,, where $E=\bigl(M\times(U_1\otimes U_{k-1})\bigr)\otimes F_M$ and $\r:E\to TM$ is induced by the projection 
$U_1\otimes U_{k-1}=U_{k-2}\oplus U_k\to U_k$\,. Therefore $\r$ admits a canonical section through which $TM$ becomes a subbundle of $E$\,. 
Thus, the Obata connection $\nabla^M$ induces, by restriction, a (classical) connection on $M$ which we shall denote by $D$.\\ 
\indent  
Recall \cite{LouPan} that a function $f$, locally defined on $M$, is called \emph{harmonic (with respect to $D$ and $[h]$)}  
if $\trace_h(D\!\dif\!f)=0$\,. 

\begin{prop} \label{prop:har_functions} 
For any $z$ and for any function $w$ locally defined on $\chi^{-1}(z)$ the function $w\circ\psi_z$ is harmonic. 
\end{prop} 
\begin{proof} 
Firstly, we shall rewrite the harmonicity equation in a convenient way. For this, we shall, also, denote by $[h]$ the linear conformal structure on $E$ 
with respect to which $TM$ and ${\rm ker}\r\,\bigl(=(M\times U_{k-2})\otimes F_M\bigr)$ are orthogonal onto each other, and 
whose restriction onto the latter is induced by $h_{k-2}$\,.\\ 
\indent 
Note that, for any (local) function $f$ on $M$ we have that $\nabla^M\!\dif\!f$ is a section of $E^*\otimes T^*M$. Consequently, 
$\b_f=({\rm Id}_{E^*}\otimes\r^*)(\nabla^M\!\dif\!f)$ is a section of $\otimes^2E^*$. Then, as $\b_f(V,V)=0$ for any $V\in{\rm ker}\r$\,,  
we obtain that $f$ is harmonic if and only if $\trace_h(\b_f)=0$\,.\\ 
\indent 
Let $\widetilde{z}\in\C\!P^1\setminus\{z\}$ and let $(s_1,\ldots,s_k,s_{\tilde{1}},\ldots,s_{\tilde{k}})$ be a local frame of $E$ with the following properties:\\ 
\indent 
\quad1) $(s_1,\ldots,s_k)$ and $(s_{\tilde{1}},\ldots,s_{\tilde{k}})$ are local frames of $\bigl(M\times(\{\ell\}\otimes U_{k-1})\bigr)\otimes F_M$ and 
$\bigl((M\times(\{\widetilde{\ell}\,\}\otimes U_{k-1})\bigr)\otimes F_M$\,, respectively, where $\ell,\widetilde{\ell}\in U_1$ are such that $z=[\ell]$ 
and $\widetilde{z}=[\widetilde{\ell}\,]$\,;\\ 
\indent 
\quad2) the orthogonal `complements' of $\r(s_1)$ and $\r(s_{\tilde{1}})$ are ${\rm ker}\dif\!\psi_z$ and ${\rm ker}\dif\!\psi_{\tilde{z}}$\,, 
respectively (equivalently, under the embedding $TM\subseteq E$ we have that (the images of) $s_1$ and $s_{\tilde{1}}$ are contained by 
${\rm ker}\dif\!\psi_z$ and ${\rm ker}\dif\!\psi_{\tilde{z}}$\,, respectively);\\ 
\indent 
\quad3) $\r\circ s_a$ and $\r\circ s_{\tilde{a}}$ are contained by the intersection of ${\rm ker}\dif\!\psi_z$ and ${\rm ker}\dif\!\psi_{\widetilde{z}}$\,, 
for any $a=1,\ldots,k$\,.\\ 
\indent 
Then the only components of $h$\,, with respect to $(s_1,\ldots,s_k,s_{\tilde{1}},\ldots,s_{\tilde{k}})$\,, which may be nonzero 
are $h_{a\tilde{b}}=h(s_a,s_{\tilde{b}})$ and $h_{\tilde{a}b}=h_{b\tilde{a}}$\,, with $a,b=1,\ldots,k$\,. 
Denote by $\bigl(h^{a\tilde{b}}\bigr)_{a,b}$ and $\bigl(h^{\tilde{a}b}\bigr)_{a,b}$ 
the inverses of the matrices $\bigl(h_{\tilde{a}b}\bigr)_{a,b}$ and $\bigl(h_{a\tilde{b}}\bigr)_{a,b}$\,, respectively.\\ 
\indent 
Denote $f=w\circ\psi_z$ and note that ${\rm ker}\dif\!f$ contains ${\rm ker}\dif\!\psi_z$ (restricted to the open set where $f$ is defined). 
Then $\trace_h(\b_f)=h^{a\tilde{b}}\b_f(s_a,s_{\tilde{b}})+h^{\tilde{a}b}\b_f(s_{\tilde{a}},s_b)$\,.\\ 
\indent 
Now, by using Proposition \ref{prop:charact_Obata} we deduce that $\b_f(s_{\tilde{a}},s_b)=0$\,, for any $a$ and $b$\,, 
and $\b_f(s_a,s_{\tilde{b}})=0$\,,  if $b\geq2$\,. Also, $h^{a\tilde{1}}=0$\,, for any $a\geq2$\,.\\ 
\indent 
Thus, $\trace_h(\b_f)=h^{1\tilde{1}}\b_f(s_1,s_{\tilde{1}})$\,. But, by using, again, Proposition \ref{prop:charact_Obata}\,, we obtain 
\begin{equation*} 
\begin{split} 
\b_f(s_1,s_{\tilde{1}})&=\bigl(\nabla^M_{s_1}\dif\!f\bigr)(\r\circ s_{\tilde{1}})\\ 
&=(\r\circ s_1)\bigl((\r\circ s_{\tilde{1}})(f)\bigr)-\bigl(\nabla^M_{s_1}(\r\circ s_{\tilde{1}})\bigr)(f)\\ 
&=(\r\circ s_1)\bigl((\r\circ s_{\tilde{1}})(f)\bigr)-\bigl[\r\circ s_1,\r\circ s_{\tilde{1}}\bigr](f)\\  
&=(\r\circ s_{\tilde{1}})\bigl((\r\circ s_1)(f)\bigr)=0\;.  
\end{split} 
\end{equation*} 
\indent 
The proof is complete. 
\end{proof} 

\indent 
In the following result $k\geq4$ (is even), and $\phi:(P,g)\to M$ is as in Theorem \ref{thm:hK_Z_surface}\,. 

\begin{cor} \label{cor:hmhK_1} 
The map $\phi:(P,g)\to(M,[h],D)$ is a harmonic morphism; that is, it pulls back harmonic functions to harmonic functions. 
In particular, if $M$ is an open subset of $U_k$ then $\phi$ is a harmonic morphism between Riemannian manifolds. 
\end{cor}  
\begin{proof}  
We already know that $\phi$ is horizontally conformal. Thus, it is sufficient to show that $\phi$ is harmonic; that is, 
$\trace_g(\nabla\otimes D)(\dif\!\phi)=0$ (see \cite{LouPan}\,).\\ 
\indent 
With the same notations as in Proposition \ref{prop:har_functions}\,, at each $x\in M$, the differentials 
of the harmonic functions $w\circ\psi_z$\,, $z\in\C\!P^1$, generate $T_x^*M$. This fact, together with an application 
of the chain rule show that $\phi$ is harmonic.\\ 
\indent 
For the last statement, note that the Obata $\r$-connection of $U_k$ is given by the trivial connection on $TU_k=U_k\times U_k$\,.   
Thus, $\phi:(P,g)\to(U_k,h_k)$ is a harmonic morphism between Riemannian manifolds. 
\end{proof} 

\indent 
The last statement of Corollary \ref{cor:hmhK_1} still holds for $k=2$\,; then $\phi$ is given by the Gibbons--Hawking construction. 
Note that, up to the fact that $g$ is compatible with the hypercomplex structure of $P$, 
our approach works in this case as well.

\section{Pluriharmonic functions and hypercomplex manifolds} \label{section:pluriharmonic_and_hypercomplex} 

\indent 
In this section, we consider only hypercomplex objects, although some of the facts hold in the more general setting of $\r$-hypercomplex manifolds.\\ 
\indent 
Let $(M,E,\r,c)$ be a hypercomplex object with twistor space $Z$ given by the surjective submersion $\psi:Y=M\times\C\!P^1\to Z$. 
Recall that then the normal bundles of the twistor spheres are positive, and we assume that the projection $Y\to\C\!P^1$ factorises as $\chi\circ\psi$\,, 
where $\chi:Z\to\C\!P^1$ is a surjective submersion whose restriction to each twistor sphere is a diffeomorphism. 
Therefore, for any $z\in\C\!P^1$, the restriction of $\psi$ to $M\times\{z\}$ defines a surjective submersion $\psi_z:M\to\chi^{-1}(z)$\,.\\ 
\indent 
Let $\g$ be a (parametrized) great circle on the sphere, which will be fixed throughout this section.  

\begin{defn} \label{defn:pluriharmonic} 
A function $u$ defined on some open set $U\subseteq M$ is called \emph{pluriharmonic} if there exists a germ of a function 
$h$ along $\chi^{-1}({\rm im}\g)$ such that, for any $x\in U$,    
\begin{equation} \label{e:defn_pluriharmonic} 
u(x)=\int_Ih(\psi_{\g(t)}(x))\dif\!t\;. 
\end{equation}  
\end{defn}  

\indent 
The sheaf (of germs) of pluriharmonic functions is just the image of the Penrose transform, with respect to $\chi^*(\ol(-1))$\,. 
Consequently, if $M$ is of constant type (equivalently, any two twistor spheres have isomorphic normal bundles), 
the sheaves of pluriharmonic and quaternionic pluriharmonic functions are equal 
(consequence of \cite{Pan-qPt}\,, where, also, the definition of the latter sheaf can be found).\\ 
\indent 
In this section, $\mathcal{G}$ will denote an abelian Lie group given by a vector space; in particular, 
its Lie algebra is $\mathcal{G}$ and the exponential map is ${\rm Id}_{\mathcal{G}}$\,.\\ 
\indent 
Let $u:M\to\mathcal{G}$ be pluriharmonic. Then, with the same notations as in Remark \ref{rem:cohomology_for_pluriharmonic}\,, 
at least locally, we may suppose that $u$ is given by a $\mathcal{G}$-valued function $h$ defined on an open set of the form 
$\chi^{-1}(D_0\cap D_{\infty})$\,, for suitably chosen open disks $D_0$ and $D_{\infty}$\,. Then $h$ gives the cocycle defining 
a principal bundle $(\mathcal{P},Z,\mathcal{G})$\,, which we call \emph{the principal bundle defined by $u$}\,.  
(Note that, a suficient condition for $\mathcal{P}$ to depend only of $u$ 
is that the Penrose transform be injective. In general, the reader may use the notation $u_h$ for the pluriharmonic function 
given by the germ $h$\,.)\\ 
\indent 
In the following results, we shall use the notations given by Remark \ref{rem:cohomology_for_pluriharmonic}\,. 

\begin{prop} \label{prop:pluri_bundle} 
Any principal bundle over $Z$ with structural group $\mathcal{G}$ is, locally, defined by a (nonunique) $\mathcal{G}$-valued 
pluriharmonic function on $M$.\\ 
\indent 
Moreover, let $(\mathcal{P},Z,\mathcal{G})$ be a principal bundle defined by the $\mathcal{G}$-valued pluriharmonic function $u$\,.  
Then $({\rm Id}_{\mathcal{G}}\otimes p)(\dif\!u)=c|_{{\rm ker}\r}$\,, 
where $c$ is the anti-self-dual principal $\r$-connection of the principal bundle on $M$ corresponding to $\mathcal{P}$, through the Ward tranform. 
\end{prop} 
\begin{proof} 
Let $(\mathcal{P},Z,\mathcal{G})$ be a principal bundle. Let $x\in M$ and let $t_x\subseteq Z$ be the corresponding twistor sphere. 
By using \cite{Siu-Stein_nbds}\,, we can find two open subsets $\mathcal{D}_0$ and $\mathcal{D}_{\infty}$ of $Z$ which are Stein 
and, with the same notations as in Remark \ref{rem:cohomology_for_pluriharmonic}\,, applied to $t_x=\C\sqcup\{\infty\}$\,, we have: 
$D_0\subseteq t_x\cap\mathcal{D}_0$ and $D_{\infty}\subseteq t_x\cap\mathcal{D}_{\infty}$\,.\\ 
\indent 
By restricting to an open neighbourhood of $x$ such that the twistor spheres corresponding to its points are contained by 
$\mathcal{D}_0\cup\mathcal{D}_{\infty}$\,, we may suppose that $Z=\mathcal{D}_0\cup\mathcal{D}_{\infty}$\,. Moreover, 
(by further restricting that neighbourhood of $x$\,), we may suppose that for any twistor sphere $t$ we have that $t\cap\chi^{-1}({\rm im}\g)$ 
is contained by $\mathcal{D}_0\cap\mathcal{D}_{\infty}$\,.\\ 
\indent 
Then, with respect to the open covering $\{\mathcal{D}_0,\mathcal{D}_{\infty}\}$\,, the principal bundle $\mathcal{P}$ is given by 
a map $h:\mathcal{D}_0\cap\mathcal{D}_{\infty}\to\mathcal{G}$. Now, we define $u:M\to\mathcal{G}$ by using \eqref{e:defn_pluriharmonic}\,, 
thus, proving the first statement.\\ 
\indent 
Note that, because $\mathcal{G}$ is abelian, $\dif\!h$ gives a cocycle for the obstruction \cite{At-57} to the existence of a principal connection 
on $(\mathcal{P},Z,\mathcal{G})$\,; denote by $o$ this obstruction.  
As the restriction of $\mathcal{P}$ to any twisor sphere $t$ is trivial, from Proposition \ref{prop:appendix} we obtain that $o|_t$ is given by the 
exact sequence of \eqref{e:twistor_spheres_in_P_and_M} applied to $t$ (and with $G=\mathcal{G}$). 
Now, the second statement follows quickly from Theorem \ref{thm:infinitesimal_Ward_transform} 
and the proof of Theorem \ref{thm:Ward_transform}\,.  
\end{proof} 

\begin{cor} \label{cor:hypercomplex_from_pluriharmonic} 
Let $u:M\to\mathcal{G}$ be pluriharmonic, and such that $({\rm Id}_{\mathcal{G}}\otimes p)(\dif\!u_x)\in{\rm End}\bigl({\rm ker}\r_x\bigr)$ is invertible, 
at some $x\in M$.\\ 
\indent 
Then, by passing, if necessary, to an open neighbourhood of $x$, the principal bundle defined by $u$ is the twistor space of a hypercomplex manifold. 
\end{cor}  
\begin{proof} 
This is an immediate consequence of Proposition \ref{prop:pluri_bundle} and Theorem \ref{thm:infinitesimal_Ward_transform}\,.  
\end{proof} 

\indent 
Let $(U,E,\r)$ be a $\r$-quaternionic vector space, with $\r$ surjective, and let $\mathcal{U}$ be its vector bundle. 
Obviously, $U$ is a hypercomplex object with twistor space $\mathcal{U}$\,. Note that, in this case, $\chi:\mathcal{U}\to\C\!P^1$ is the bundle projection. 
Also, recall that, there exists a canonical isomorphism $H^1\bigl(\mathcal{U}^*\bigr)=({\rm ker}\r)^*$.\\ 
\indent 
In the following result, we make no distinction between a principal bundle and its equivalence class.  

\begin{cor} \label{cor:hmhK_2} 
Let $M$ be an open subset of $U$, and let $Z\subseteq\mathcal{U}$ be its twistor space. Denote by $\mathcal{G}$ the abelian group $({\rm ker}\r\,,+)$\,.\\ 
\indent 
Then there exists a natural correspondence between the following:\\ 
\indent 
\quad{\rm (i)} Hypercomplex manifolds whose twistor spaces are principal bundles over $Z$ with structural group $\mathcal{G}$.\\ 
\indent 
\quad{\rm (ii)} Elements of $H^1\bigl(Z,\chi^*\bigl(\mathcal{U}^*\bigr)\bigr)\otimes\mathcal{G}$ whose restrictions to each twistor sphere are invertible.\\ 
\indent 
Furthermore, if $\mathcal{U}=\ol(k)$\,, with $k\geq2$ even, then any hypercomplex manifold $P$ obtained through this correspondence 
may be endowed with a hyper-K\"ahler metric such that the projection $P\to M$ be a harmonic morphism. 
\end{cor} 
\begin{proof} 
By using \cite[Corollary]{Buch-rel_deR} we obtain, similarly to \cite[\S5]{PePo-88}\,, a natural isomorphism 
$H^1\bigl(Z,\mathcal{G}\bigr)=H^1\bigl(Z,\chi^*\bigl(\mathcal{U}^*\bigr)\bigr)\otimes\mathcal{G}$\,. Then the correspondence follows 
quickly from (the proof of) Proposition \ref{prop:pluri_bundle}\,.\\ 
\indent 
The last statement is an immediate consequence of Corollary \ref{cor:hmhK_1}\,. 
\end{proof}

\appendix 

\section{} 

\indent 
Here we prove a result used in the proof of Proposition \ref{prop:pluri_bundle}\,. Recall that, we work in the complex-analytic category.\\ 
\indent 
Let $(P,M,G)$ be a principal bundle, and let $\p:P\to M$ be its projection. Denote by $o\in H^1\bigl(M,{\rm Hom}(TM,{\rm Ad}P)\bigr)$ 
the obstruction \cite{At-57} to the existence of a principal connection on $P$.\\ 
\indent 
Let $N\subseteq M$ be a submanifold and suppose that there exists $s:N\to P$ a section of $P|_N$\,. On identifying $N$ and $s(N)$ we have the following 
exact sequence of vector bundles over $N$  
\begin{equation} \label{e:appendix_normal_N}
0\longrightarrow({\rm ker}\dif\!\p)|_{s(N)}\longrightarrow\mathcal{N}_PN\longrightarrow\mathcal{N}_MN\longrightarrow0\;,  
\end{equation} 
where $\mathcal{N}_PN$ and $\mathcal{N}_MN$ are the normal bundles of $N$ in $P$ and $M$, respectively. 
Denote by $o_N\in H^1\bigl(N,{\rm Hom}\bigl(\mathcal{N}_MN,({\rm Ad}P)|_N\bigr)\bigr)$ the cohomology class of \eqref{e:appendix_normal_N}\,.   

\begin{prop} \label{prop:appendix} 
If $j^*(o|_N)=0$ then $o_N$ is the unique cohomology class 
such that $k^*(o_N)=o|_N$\,, where $j:TN\to TM|_N$ and $k:TM|_N\to\mathcal{N}_MN$ are the canonical morphisms of vector bundles. 
\end{prop} 
\begin{proof} 
We have the following commutative diagram 

\begin{displaymath} 
\xymatrix{ 
                  &                       0                                    &                        0                                &                       0                                  &          \\
  0  \ar[r]   &  N\times\mathfrak{g}   \ar[r] \ar[u]    &  \mathcal{N}_PN  \ar[r] \ar[u]        &   \mathcal{N}_MN \ar[r] \ar[u]      &     0   \\ 
   0  \ar[r]   &  N\times\mathfrak{g}   \ar[r] \ar[u]^=    &  TP|_{s(N)}   \ar[r] \ar[u]                &   TM|_N                   \ar[r] \ar[u]_k        &     0   \\ 
                  &                       0               \ar[r] \ar[u]   &        TN          \ar[r]^= \ar[u]            &   TN                         \ar[r] \ar[u]_j       &    0    \\
                  &                                                             &          0                         \ar[u]            &    0                                    \ar[u]        &          }
\end{displaymath} 
where we have used the isomorphism $({\rm ker}\dif\!\p)|_{s(N)}=N\times\mathfrak{g}$\,, induced by the fact that $P|_N$ is trivial. 
Now, on dualizing the diagram and then by passing to the cohomology sequences of the rows we obtain the commutative diagram 
\begin{displaymath} 
\xymatrix{ 
           0   \ar[d]    &                       0    \ar[d]                                         \\
  H^1\bigl(N,\mathcal{N}_M^*N\bigr) \ar[d]^{k^*}   &  H^0(N,N\times\mathfrak{g}^*)   \ar[l]_{o_N} \ar[d]^=       \\ 
   H^1\bigl(N,T^*M|_N\bigr)  \ar[d]^{j^*}   &  H^0(N,N\times\mathfrak{g}^*)   \ar[l]_{o|_N} \ar[d]      \\ 
   H^1(N,T^*N)        \ar[d]              &       0          \ar[l]                  \\ 
          0        &                                    } 
\end{displaymath} 
from which the result quickly follows. 
\end{proof}


\begin{thebibliography}{10} 


\bibitem{At-57} 
M.~F.~Atiyah, Complex analytic connections in fibre bundles, 
\textit{Trans. Amer. Math. Soc.}, {\bf 85} (1957) 181--207. 

\bibitem{Buch-rel_deR} 
N.~Buchdahl, On the relative de Rham sequence, 
\textit{Proc. Amer. Math. Soc.}, {\bf 87} (1983) 363--366. 

\bibitem{BaiWoo2}
P.~Baird, J.~C.~Wood, \textit{Harmonic morphisms between Riemannian manifolds},
London Math. Soc. Monogr. (N.S.), no. 29, Oxford Univ. Press, Oxford, 2003. 

\bibitem{GibHaw}
G.~W.~Gibbons, S.~W.~Hawking, Gravitational multi-instantons,
\textit{Phys. Lett. B}, {\bf 78} (1978) 430--432.

\bibitem{Hit-monop_geod} 
N.~J.~Hitchin, Monopoles and geodesics. 
\textit{Comm. Math. Phys.}, {\bf 83} (1982) 579--602. 

\bibitem{LouPan} 
E.~Loubeau, R.~Pantilie, Harmonic morphisms between Weyl spaces and twistorial maps, 
\textit{Comm. Anal. Geom.}, {\bf 14} (2006) 847--881. 

\bibitem{fq}
S.~Marchiafava, L.~Ornea, R.~Pantilie, Twistor Theory for CR quaternionic manifolds and related structures,
\textit{Monatsh. Math.}, {\bf 167} (2012) 531--545.  

\bibitem{fq_2}
S.~Marchiafava, R.~Pantilie, Twistor Theory for co-CR quaternionic manifolds and related structures,
\textit{Israel J. Math.}, {\bf 195} (2013) 347--371.  

\bibitem{Pan-EAR} 
R.~Pantilie, \textit{Submersive harmonic maps and morphisms}, 
Editura Academiei Rom\^ane, Bucure\c sti, 2009.

\bibitem{Pan-twistor_(co-)cr_q} 
R.~Pantilie, On the twistor space of a (co-)CR quaternionic manifold, 
\textit{New York J. Math.}, {\bf 20} (2014) 959--971.

\bibitem{Pan-qgfs} 
R.~Pantilie, On the embeddings of the Riemann sphere with nonnegative normal bundles, Preprint IMAR, Bucharest, 2013, 
(available from \href{http://arxiv.org/abs/1307.1993}{\tt http://arxiv.org/abs/1307.1993}). 

\bibitem{Pan-qPt} 
R.~Pantilie, The Penrose transform in quaternionic geometry, Preprint IMAR, Bucharest, 2014, 
(available from \href{http://arxiv.org/abs/1411.2226}{\tt http://arxiv.org/abs/1411.2226}). 

\bibitem{PePo-88} 
H.~Pedersen, Y.~S.~Poon, Hyper-K\"ahler metrics and a generalization of the Bogomolny equations, 
\textit{Comm. Math. Phys.}, {\bf 117} (1988) 569--580. 

\bibitem{Qui-QJM98}
D.~Quillen, Quaternionic algebra and sheaves on the Riemann sphere,
\textit{Q. J. Math.}, {\bf 49} (1998) 163--198. 

\bibitem{Siu-Stein_nbds} 
Y.~T.~Siu, Every Stein subvariety admits a Stein neighborhood, 
\textit{Invent. Math.}, {\bf 38} (1976/77) 89--100.



\end{thebibliography}
\end{document}